\documentclass[11pt, reqno]{amsart}
\usepackage{version}
\usepackage{amssymb}
\usepackage[]{enumitem}

\makeatletter
\newcommand{\pushright}[1]{\ifmeasuring@#1\else\omit\hfill$\displaystyle#1$\fi\ignorespaces}
\newcommand{\pushleft}[1]{\ifmeasuring@#1\else\omit$\displaystyle#1$\hfill\fi\ignorespaces}
\makeatother

\usepackage{hyperref}
\usepackage{lipsum}

\newcommand{\myfootnote}[1]{
    \renewcommand{\thefootnote}{}
    \footnotetext{\hspace{-2pt}\scriptsize#1}
    \renewcommand{\thefootnote}{\arabic{footnote}}
}

\newcommand*{\doi}[1]{\href{http://dx.doi.org/#1}{DOI: #1}}

\theoremstyle{definition}
\newtheorem{thm}{Theorem}[section]

\newtheorem{lem}{Lemma}[section]
\newtheorem{prop}{Proposition}[section]

\newtheorem{rk}{Remark}[section]

\usepackage{mathrsfs}
\usepackage{amsmath, amsfonts, amssymb,amsthm}
\newcommand{\field}[1]{\mathbb{#1}}
\newcommand{\N}{\field{N}}

\newcommand{\R}{\field{R}}
\newcommand{\C}{\field{C}}

\DeclareMathOperator{\id}{Id}

\DeclareMathOperator{\dist}{dist}

\DeclareMathOperator{\cl}{cl}

\numberwithin{equation}{section}

\begin{document}
\title[]{Extended explanation of Orevkov's paper on proper holomorphic embeddings of complements of Cantor sets in $\mathbb C^2$ and a discussion of their measure.}

\author{Giovanni D. Di Salvo}
\begin{abstract}
We clarify the details of a cryptical paper by Orevkov in which a construction of a proper holomorphic embedding $\gamma\colon\Bbb P^1\setminus C\hookrightarrow\C^2$ is presented; in particular, it is proved that such a construction can be done to get the Cantor set $C$ to have zero Hausdorff dimension.
\end{abstract}
\myfootnote
{
	Available on arXiv, \doi{10.48550/arXiv.2211.05046v2}.
}

\maketitle
\section{Introduction}

\subsection{The main result and motivation}
A very relevant open problem in complex geometry is to investigate whether every open Riemann surface $\mathcal R$ admits a proper holomorphic embedding into $\Bbb C^2$.
It has been suggested to search for a counterexample as the complement of a "thick" Cantor set $C$ inside the Riemann Sphere, that is $\mathcal R=\overline \C\setminus C$.
In paper \cite{O}, the author proved the existence of a proper holomorphic embedding from the complement of a Cantor set $C\subset\overline\C$ into $\C^2$.
Nevertheless no information dealing with the size of $C$ is given; moreover such a paper is quite cryptical. The motivations for the present paper are therefore two: to write in detail the construction of both the Cantor set $C\subset\overline\C$ and the proper holomorphic embedding $\gamma \colon \overline\C\setminus C\hookrightarrow\C^2$ originally presented in \cite{O} and discussing the measure of $C$, proving in particular that it can be constructed to have zero Hausdorff measure. From the construction it will be rather clear that it is really hard to prove that this $C$ could have any kind of positive measure.
The main result of the present paper, originally presented by S. Orevkov in \cite{O} without the estimate on the size of $C$, is the following

\begin{thm}[]\label{main}
There exists a Cantor set $C\subset \overline\C$ with $\dim_H(C)=0$ and a 
proper holomorphic embedding $\gamma\colon\overline\C\setminus C\hookrightarrow\mathbb C^2$. 
\end{thm}
It worth mentioning that in \cite{WDS} has been proved that the complement of a thick Cantor set inside $\overline \C$ does not constitute a counterexample, as it is therein provided the construction of a proper holomorphic embedding $\overline\C\setminus C\hookrightarrow\mathbb C^2$ where $C\subset\overline\C$ is a Cantor set of arbitrarily large measure. Namely: such $C$ is realized as a subset of a square whose side has length $2$ and for every $\epsilon>0$ it can be built of Lebesgue measure greater than $4-\epsilon$.

\section{Definitions and preliminary results}\label{prelim}

\subsection{Cantor sets}
The Cantor set is defined as the intersection of the decreasing family of subsets of the unit interval of the real line $E_0=[0,1]\subset\R$ defined inductively by removing at each step the open middle third from each interval of the previous set.\\
Thus $E_1=[0,\frac13]\cup[\frac23,1]$, $E_2=[0,\frac19]\cup[\frac29,\frac13]\cup[\frac23,\frac79]\cup[\frac89,1]$ and so on. Finally one defines
$$
E:=\bigcap_{n\ge0}E_n\subset[0,1].
$$
Given a topological space $X$, we define \emph{Cantor set} any subset $C\subset X$ homeomorphic to $E$. Hence a Cantor set is characterized by the topological properties of $E$; thus, according to Brouwer's theorem in \cite{B}, $C\subset X$ is a Cantor set if and only if it is\\
\begin{enumerate}[label=(\roman*)$_{C}$]
\item\label{ne} not-empty\;;
\item\label{per} perfect (i.e. closed and with no isolated points)\;;
\item\label{com} compact\;;
\item\label{td} totally disconnected (i.e. every its connected component is a one-point set)\;;
\item\label{met} metrizable\;.
\end{enumerate}

\subsection{Sperichal Distance}\label{sph}

The one point compactification of $\R^n$ is homeomorphic to $\Bbb S^n=\{(x_1,\dots,x_{n+1})\in\R^{n+1}\;:\;x_1^2+\cdots+x_{n+1}^2=1\}$ via stereographic projection:
\begin{align*}
\varphi_n&\colon\overline{\R^n}\to\Bbb S^n\;,\\
\varphi_n(x)&=\frac2{\|x\|^2+1}\left(x,\frac{\|x\|^2-1}{2}\right),\;\;x\neq\infty\;;\\
\varphi_n(\infty)&=(0,\dots,0,1)\;.
\end{align*}
This allows us to define a notion of distance on any such $\overline{\R^n}$ as
$$
d_n(x,y):=\arccos(\varphi_n(x)\cdot\varphi_n(y)),\;\;\;x,y\in\overline{\Bbb R^n}\;\;,
$$
(see e.g. \cite{BEG}, Theorem 5 pag. 444) where $\cdot$ denotes the scalar product in $\Bbb R^{n+1}$ . The spaces $\overline{\C}$ and $\overline{\C^2}$ are just the case $n=2$ and $n=4$ respectively, so we have a notion of distance on them, allowing us to talk about diameter of their subsets. Namely, for a connected subset $\Omega\subseteq\overline{\R^n}$ we denote its diameter as
$
|\Omega|:=\sup\{d_n(x,y)\;:\;x,y\in\Omega\}
$ and if $\Omega=\bigcup_j\omega_j$, where $\omega_j$ are pairwise disjoint connected subsets of $\overline{\R^n}$ then we define $|\Omega|:=\sup_j|\omega_j|$.

\subsection{Definitions and technical results}\label{def}
We will consider the Riemann Sphere $\overline \C:=\C\cup\{\infty\}$ and the one-point compactification of $\C^2$ as $\overline{\C^2}:=\C^2\cup\{\infty_2\}$.\\
\newline
Given a subset $B\subset\C^2$ and a positive number $\alpha>0$, we define the open subset
$$
B(\alpha):=\{z\in\C^2\;:\;\dist(z,B)<\alpha\}\;.
$$
Given a subset $\Omega$ of $\overline \C$ (resp. $\overline{\C^2}$), we say that $\Omega$ is \emph{bounded} if $\Omega\cap U=\emptyset$ for some $U$ open neighborhood of $\infty$ (resp. of $\infty_2$), \emph{unbounded} otherwise.\\
\newline
We consider $\{\epsilon_n\}_{n\ge1}, \{R_n\}_{n\ge0}$ strictly monotone sequences of positive real numbers, with $\epsilon_n\to0$ and $R_n\to+\infty$. Consider moreover $\{a_n\}_{n\ge1}\subset\C$ such that $R_{n-1}<|a_n|<R_n$.\\
We will refer to these sequences as the \emph{parameters} of the construction.\\
Fixing these sequences, the sets
\begin{align*}
D_n:=&\{|z|<R_n\}\subset\C\;\;\\
C_n^v:=&D_n\times\C\;\;\\
C_n^h:=&\C\times D_n\;\;\\
C_n:=&C_n^v\cup C_n^h \;\;\\
B_n:=&\overline{C_n^v\cap C_n^h}=\overline{D}_n\times \overline{D}_n \;\;\\
\end{align*}
and the rational shears $f_n\colon\C^2\rightarrow\overline{\C^2}$\;,
\begin{align*}
f_n(x,y):=
	\begin{cases}
		\left(x,y+g_n(x)\right)\;\;\mbox{for odd}\;\;n\;\;\\
		\left(x+g_n(y),y\right)\;\;\mbox{for even}\;\;n\;\;,
	\end{cases}
\end{align*}
where
$$
g_n(z):=\frac{\epsilon_n}{z-a_n}\;,
$$
are defined. Finally define the sequence of continuous mappings $\gamma_n\colon\overline\C\to\overline{\C^2}$ as
\begin{align*}
\gamma_n(z)&:=f_n\circ\cdots\circ f_1(z,0),\;\;z\neq\infty\\
\gamma_n(\infty)&:=\infty_2
\end{align*}
and consequently the sets
$$
A_n:=\gamma_n(\overline\C)=f_n\circ\cdots\circ f_1(\C\times\{0\})\cup\{\infty_2\}\;.
$$
The construction of the proper holomorphic embedding $\gamma$ and of the Cantor set $C$ in Theorem \ref{main} relies on the sets and functions just defined: we will see that a suitable choice of the sequences $\{\epsilon_n\}_n,\;\{R_n\}_n$ and $\{a_n\}_n$ allows to obtain $C$ as the closure of the set $\bigcup_n\gamma_n^{-1}(\infty_2)$ and the limit $\lim_n\gamma_n$ to be a proper holomorphic embedding $\gamma$ on $\overline\C\setminus C$ into $\C^2$; in particular the faster $\{R_n\}_n$ diverges, the smaller the size of $C$, so that it can be constructed to have zero Hausdorff dimension.\\
\newline
In the rest of this paper $\{k_n\}_{n\ge0}$ will denote the Fibonacci sequence, that is the sequence of natural numbers defined by $k_0=k_1=1$ and
$$
k_{n+2}=k_{n+1}+k_n,\;\;\; \forall n\ge0\;.
$$
The following remark gathers some known fact about the number of solutions of rational equations (see 4.23--4.25, \cite{F}), that will be fundamental in the proof of the subsequent Lemma.
\begin{rk}\label{rational}
Let $r\colon\overline\C\to\overline\C$ be a rational function. There exists a natural number $k$, called \emph{degree} of $r$ and denoted as $\deg r$, such that for every $a\in\overline\C$, the equation $r(z)=a$ has $k$ solutions counted with their multiplicity, which are multiple only for a finite number of $a\in\overline\C$.
\begin{enumerate}
\item \label{zerisemplici}
In particular, up to a slightly small perturbation of $a$, the zeros of $r(z)-a$ are simple.
\item\label{polisemplici}
The zeros of $r$ are as many as the poles, counted with their multiplicity, therefore the poles of $r$ are all simple if and only if they are $k$ distinct.
\item\label{disjointsolutions}
Let $t$ be another rational function and $w\in\overline{\C}$ fixed. Up to perturb $a\in\C$ arbitrarily little, the equations $t=w$ and $r-a=0$ have no common solutions, as it is easily seen observing that $r^{-1}(a)\cap r^{-1}(a')=\emptyset$ for any $a'\neq a$.
\end{enumerate}
\end{rk}
\begin{lem}\label{tech}
Define a sequence of rational functions recursively:
\begin{align}
r_0(z)&:=z\nonumber\\
r_1(z)&:=g_1(z)\nonumber\\
r_{n+2}(z)&:=r_n(z)+g_{n+2}(r_{n+1}(z))=r_n(z)+\frac{\epsilon_{n+2}}{r_{n+1}(z)-a_{n+2}},\;\;n\ge0.\label{recursive}
\end{align}
Then, up to perturbing the elements of any sequence $\{a_n\}_n\subset\C$ arbitrarily little, the following hold:
\begin{enumerate}[label=(\roman*)]
\item\label{unsharedpoles}
for every $n\ge0$, $r_n$ and $r_{n+1}$ do not share any pole and $r_n^{-1}(\infty)\subsetneq r_{n+2}^{-1}(\infty)$; in particular $\infty$ is a pole only for $r_n$ with $n$ even;
\item\label{numberofsimplesolutions} 
the poles of $r_n$ are $k_n$ and all simple; therefore the equation
$$
r_n(z)=a_{n+1}
$$
has exactly $k_n$ solutions, all distinct.
\end{enumerate}
\end{lem}
\begin{proof}
We see that $r_n^{-1}(\infty)\cap r_{n+1}^{-1}(\infty)=\emptyset$ for all $n\ge0$ by induction on $n$. The base case follows immediately. Assume the claim true for $n$; if by contradiction $\exists w\in r_{n+1}^{-1}(\infty)\cap r_{n+2}^{-1}(\infty)$, then by (\ref{recursive}), it must be $r_{n}(w)=\infty$, so $r_n^{-1}(\infty)\cap r_{n+1}^{-1}(\infty)\neq\emptyset$, against inductive hypothesis.\\
The inclusion $r_n^{-1}(\infty)\subsetneq r_{n+2}^{-1}(\infty)$ for all $n\ge0$ follows by looking at (\ref{recursive}) as from (\ref{disjointsolutions}) in Remark \ref{rational}, up to slightly perturbing $a_{n+2}$, one gets $\{r_n=\infty\}\cap\{r_{n+1}-a_{n+2}=0\}=\emptyset$.\\
Clearly $\infty$ is a pole for $r_0$, thus what said so far ensures that $\infty$ is a pole for $r_{2k}$ and not a pole for $r_{2k+1}$, for all $k\ge0$. So \ref{unsharedpoles} is proved.

Let us prove \ref{numberofsimplesolutions} by induction on $n$.
This is trivial for $r_0$ and $r_1$.
Assume the claim true up to $n+1$. As observed in (\ref{disjointsolutions}) in Remark \ref{rational}, up to a slightly small perturbation of $a_{n+2}$, the sets $\{r_n=\infty\}$ and $\{r_{n+1}-a_{n+2}=0\}$ have empty intersection.
Hence, by (\ref{recursive}), the poles of $r_{n+2}$ are precisely the poles of $r_n$ (which are $k_n$ and simple by inductive hypothesis) and the poles of $\frac{\epsilon_{n+2}}{r_{n+1}-a_{n+2}}$ (which are $k_{n+1}$ and simple as the zeros of the denominator are, up to a slightly small perturbation of $a_{n+2}$ as pointed out in (\ref{zerisemplici}) in Remark \ref{rational}, in fact $\deg r_{n+1}=k_{n+1}$ by inductive hypothesis). Therefore the poles of $r_{n+2}$ are $k_{n}+k_{n+1}=k_{n+2}$ and all simple, hence by (\ref{polisemplici}) in Remark \ref{rational} one has $\deg r_{n+2}= k_{n+2}$. So up to a slightly small perturbation of $a_{n+3}$, (\ref{zerisemplici}) guarantees that the equation $r_{n+2}(z)=a_{n+3}$ has $k_{n+2}$ solutions, all distinct.
So \ref{numberofsimplesolutions} is proved.
\end{proof}

We will assume from now on to choose $R_n$ large enough, so that $r_n(z)=a$ has $k_n$ distinct solutions for every $|a|\ge R_n$. The following Lemma is used in Section \ref{pp}:
\begin{lem}\label{analysis}
Let $r$ be a rational function with a finite number of poles $r^{-1}(\infty)=\{w_1,\dots,w_N\}\subset\overline\C$, all simple. Assume that $r(z)=a$ has $k$ distinct solutions for every $|a|\ge T_0$, for some $T_0$.
Then for any $T\ge T_0$ large enough, the set
\begin{equation*}
\Omega(r,T):=\{z\in\overline\C\setminus r^{-1}(\infty)\;:\;|r(z)|>T\}
\end{equation*}
is the disjoint union of $k$ connected sets, each of which homeomorphic to the punctured open disk $\{|z|<1\}\setminus\{0\}$ and containing exactly one solution of $r(z)=a$, for every $|a|>T$.
Moreover, such $k$ sets can be taken to have arbitrarily small diameters, provided that $T$ is large enough. 
\end{lem}
\begin{proof} 
Assume the poles of $r$ are all distinct; being them simple one has that $N=k$.
We express $r$ on a neighborhood $V_j\subset\overline\C$ of a pole $w_j$ as $r(z)=h_j(z)(z-w_j)^{-1}$ or $r(z)=h_j(z)z$, accordingly with $w_j$ to be different or equal to $\infty$, where $h_j\in\mathcal O(V_j)$ is bounded and never vanishing.
Let $a\in\C$ with $|a|>T$. The larger $T$, the closer $z$ has to be to the poles for $r(z)=a$ to be solved. In particular, up to enlarging $T$ we may assume that the $V_j$ are pairwise disjoint, therefore $X_j:=\{z\in V_j\setminus\{w_j\}\;:\;|r(z)|>T\}$ are pairwise disjoint and connected. In particular for each $\epsilon>0$, there exists $T>0$ sufficiently large such that the diameter of $X_j$ is less than $\epsilon$.
Moreover the union of the $X_j$ is $\Omega(r,T)$ and each $X_j$ contains precisely one of the $k$ solutions of the equation $r(z)=a$ for any $|a|>T$, for $T$ sufficiently large: if not, there must be one of them, say $X_1$, such that for any $T$ sufficiently large it contains at least two solutions of $r(z)=a$ for some $|a|>T$, leading $w_1$ to be a multiple pole for $r$.
\end{proof}

Finally, it is easily seen by induction that 
\begin{align}\label{rat}
\gamma_n=
	\begin{cases}
		\left(r_{n-1},r_{n}\right)\;\;\mbox{for odd}\;\;n\;\;\\
		\left(r_n,r_{n-1}\right)\;\;\mbox{for even}\;\;n\;\;,
	\end{cases}
\end{align}
in fact the case $n=1$ is immediate and for $n$ even, for example
$$
\gamma_{n+1}=f_{n+1}(\gamma_n)=f_{n+1}(r_n,r_{n-1})=(r_n,r_{n-1}+g_{n+1}(r_n))=(r_n,r_{n+1})\;.
$$
\vspace{0cm}
\section{General Properties}
This section is devoted to the description of the objects introduced in Section \ref{def}, which are the ingredients of the construction of the proper holomorphic embedding $\gamma\colon\overline{\C}\setminus C\hookrightarrow\C^2$ presented in \cite{O}. In particular we will focus on the choice of the three sequences $\{\epsilon_n\}_{n\ge1},\{R_n\}_{n\ge0},\{a_n\}_{n\ge1}$ upon which the whole construction is based.


\subsection{On the set $\gamma_n^{-1}(\infty_2)$ and the sequence $\{R_n\}_{n\ge0}$. Definition and properties of vertical and horizontal components of $A_n$}\label{vhcomp}
It follows from (\ref{rat}) and \ref{unsharedpoles} in Lemma \ref{tech} that
\begin{equation}\label{incl}
\gamma_{n}^{-1}(\infty_2)\subsetneq\gamma_{n+1}^{-1}(\infty_2)
\end{equation}
holds true for all $n\ge0$; moreover $\infty\in\gamma_{n}^{-1}(\infty_2)$ for every $n\ge0$, so we can write
\begin{equation}\label{elem}
\gamma_{n}^{-1}(\infty_2)=\{\infty\}\cup\{t_j^{(n)}\}_{j=1}^{b_n-1}
\end{equation}
for some $t_j^{(n)}\in\C$, where $b_n=k_{n+1}$.
Observe that $A_n=\gamma_n(\overline{\C})\subset\overline{\C^2}$ is compact with $\infty_2\in A_n$; moreover $A_n\setminus\{\infty_2\}=\gamma_n({\C})$ is a connected complex curve in $\C^2$.
\begin{prop}\label{gen}
Assume $\epsilon_1,\dots,\epsilon_n$, $a_1,\dots,a_n$ and $R_1,\dots,R_{n-1}$ fixed, $|a_{j+1}|>R_{j}>|a_{j}|$ for $j=1,\dots,n-1$. Then there exist $R_n>|a_n|$ such that what follows is true.
\begin{enumerate}[label=(\alph*)]
	\item\label{disjcomp} $A_n\setminus (B_n\cup\{\infty_2\})$ has $b_n$ disjoint components.
	\item\label{split} 
	$A_n\subset C_n$.
\end{enumerate}
We refer to the components of $A_n\setminus(B_n\cup\{\infty_2\})$ contained in $C_n^{h}$ (resp. in $C_n^{v}$) as the \emph{horizontal} (resp. \emph{vertical}) components of $A_n$.\\
Correspondingly we split $b_n$ as $h_n+v_n$. It turns out that
\begin{align}\label{vh}
  \begin{cases}
      v_n=k_{n-1}\\
      h_n=k_n
    \end{cases} (n\;\; \text{even}),\hspace{1cm}
                \begin{cases}
      			v_n=k_n\\
      			h_n=k_{n-1}
    \end{cases} (n\;\; \text{odd}).       
\end{align}
\end{prop}
\begin{proof}
The components of $A_n\setminus(B_n\cup\{\infty_2\})$ are the image, via $\gamma_n$ of punctured neighborhoods in $\overline \C$ of the points (\ref{elem}) (which are defined once $a_n,\epsilon_n,a_j,\epsilon_j,R_j,\;\;j\le n-1$ are fixed). So, by Lemma \ref{analysis}, the bigger $R_n$, the smaller these neighborhoods; hence these $b_n$ neighborhoods are disjoint for $R_n$ large enough, achieving \ref{disjcomp}, as $\gamma_n$ is injective on $\overline\C\setminus\gamma_n^{-1}(\infty_2)$ (see section \ref{sets}).\\
Looking at (\ref{rat}) it is clear that unbounded components of $A_n=\gamma_n(\overline\C)$ are due to the poles of $r_n$ and $r_{n-1}$; being $r_{n}^{-1}(\infty)\cap r_{n-1}^{-1}(\infty)=\emptyset$ by \ref{unsharedpoles} in Lemma \ref{tech}, it follows that when one component of $A_n$ becomes unbounded, the other remains bounded. Hence, for $R_n$ large enough, we get \ref{split}.\\
Finally, (\ref{vh}) follows directly from Lemma \ref{tech}, once we express $\gamma_n$ as in (\ref{rat}).
\end{proof}

\subsection{Definition of certain sets $\Delta_n,\;K_n$ and some properties of the mappings $\gamma_n$}\label{sets}

Let us the define
$$
\Delta_n:=\gamma_n^{-1}(A_n\setminus B_n)\;.
$$
Given $U\subset\overline{\C^2}$ open neighborhood of $\infty_2$, $\gamma_n^{-1}(U)$ is a neighborhood of $\gamma_n^{-1}(\infty_2)$; this last set has $b_n$ points, therefore Lemma \ref{analysis} guarantees that for $U$ sufficiently small $\gamma_n^{-1}(U)$ is the disjoint union of $b_n$ open neighborhoods $U_j^{(n)}$ of the points (\ref{elem}) (meaning the points of $U$ are sufficiently close to $\infty_2$ with respect to the spherical measure introduced in section \ref{sph}; equivalently $U$ is contained in the complement of a closed ball in $\overline{\C^2}$ of sufficiently large radius). In particular $\forall\epsilon>0\;\exists U\subset\overline{\C^2}$ open neighborhood of $\infty_2$ such that $|U_j^{(n)}|<\epsilon$.\\
Since
\begin{equation}\label{delta}
\Delta_n=\gamma_n^{-1}(\overline{\C^2}\setminus B_n)
\end{equation}
(straightforward with double inclusion), we can always suppose to have $R_n$ large enough such that
$$
\Delta_n=\bigcup_{j=1}^{b_n}U_j^{(n)}
$$
is a disjoint union. In particular there exists $R_n$ so large that
\begin{equation}\label{small}
|U_j^{(n)}|<\frac1{b_n^{n}}\;,\;\;\;\;\forall j=1,\dots,b_n\;.
\end{equation}
Let us now check that
$$
\gamma_n\colon\overline\C\setminus\gamma_n^{-1}(\infty_2)\to\C^2
$$
is a holomorphic embedding by induction on $n$. As $\gamma_1\colon\overline\C\setminus\{a_1,\infty\}\to\C^2,\;\gamma_1(z)=(z,\frac{\epsilon_1}{z-a_1})$, the base case trivially follows.
Assume the claim true for $n$. Then $\gamma_{n+1}=f_{n+1}\circ\gamma_n$ is a holomorphic embedding in the intersection of the subset where $\gamma_n$ is such (that is $\overline\C\setminus\gamma_n^{-1}(\infty_2)$) and $\{z\in\overline\C\;:\;\gamma_n(z)\notin f_{n+1}^{-1}(\infty_2)\}=\overline\C\setminus\gamma_{n+1}^{-1}(\infty_2)$. Such an intersection is
$$
\overline \C\setminus\gamma_{n}^{-1}(\infty_2)
\cap\overline \C\setminus\gamma_{n+1}^{-1}(\infty_2)
=\overline \C\setminus(\gamma_{n}^{-1}(\infty_2)\cup\gamma_{n+1}^{-1}(\infty_2))
=\overline \C\setminus\gamma_{n+1}^{-1}(\infty_2)
$$
and we are done (the last equality follows from (\ref{incl})).

Define then the following compacts in $\overline\C$:
\begin{align*}
K_n&:=\overline\C\setminus\Delta_n=\gamma_n^{-1}(A_n\cap B_n)=\gamma_n^{-1}(B_n),\;\;n\ge1\\
K_0&:=\emptyset\;.
\end{align*}
It is straighforward to see that $K_n\subset \overline\C\setminus\gamma_n^{-1}(\infty_2)$. We will see in (\ref{kincr}) that $K_n\subsetneq K_{n+1}^{\circ}$; therefore there exists $\delta_n>0$ such that if $h\colon\overline\C\setminus\gamma_n^{-1}(\infty_2)\to\C^2$ is holomorphic with 
$$
\|\gamma_n-h\|_{K_n}<\delta_n\;,
$$
then $h$ is an embedding on $K_{n-1}$. Set $\delta_0:=1/4$ and assume without loss of generality $\{\delta_n\}_{n\ge0}$ to be strictly decreasing.


\subsection{From $\Delta_n$ to $\Delta_{n+1}$: definition of pair of pants}\label{pp}
We will prove that for a suitable choice of the parameters, the sequences $\{\Delta_n\}_n$ and $\{K_n\}_n$ are strictly decreasing and increasing respectively; in particular we will focus on the topological behavior of the sequence $\{\Delta_n\}_n$, when passing from step $n$ to step $n+1$.\\
\begin{prop}\label{prop:monotone}
Assume $\epsilon_1,\dots,\epsilon_{n+1}$, $a_1,\dots,a_{n+1}$ and $R_1,\dots,R_{n}$ fixed, with $|a_{j+1}|>R_{j}>|a_j|$ for $j=1,\dots,n$ and rename the $U_j^{(n)}$ to highlight which one comes from horizontal or vertical components of $A_n\setminus B_n$ as
$$
\Delta_n=\bigcup_{j=1}^{h_n}H_j^{(n)}\;\;\cup\;\;\;\bigcup_{j=1}^{v_n}V_j^{(n)}=H_n\cup V_n\;.
$$
Then, for any $R_{n+1}>|a_{n+1}|$ sufficiently large, we have that
\begin{align}
\label{deltadecr}&\Delta_{n+1}\subsetneq\Delta_n\\
\label{kincr}&K_{n}\subsetneq K_{n+1}^{\circ}\;
\end{align}
and passing from $\Delta_n$ to $\Delta_{n+1}$, for $n$ even, $V_j^{(n)}$ shrinks to $V_j^{(n+1)}$, $j=1,\dots,v_n$, while $H_j^{(n)}$ splits itself into $H_j^{(n+1)}$ and $V_{v_n+j}^{(n+1)}$, $j=1,\dots,h_n$, creating the so called \emph{pair of pants}.
For $n$ odd, a similar conclusion holds, swapping vertical and horizontal components.
\end{prop}
\begin{proof}
Let $n$ be even. 
From (\ref{recursive}) and (\ref{rat}) we have 
\begin{align*}
\gamma_n&=(r_n,r_{n-1})\\
\gamma_{n+1}&=(r_n,r_{n+1})=(r_n,r_{n-1})+\left(0,\frac{\epsilon_{n+1}}{r_n-a_{n+1}}\right)\;,
\end{align*}
from which 
\begin{align*}
H_n&=\{|r_n|>R_n\},\;\;\\
V_n&=\{|r_{n-1}|>R_n\},\\
H_{n+1}&=\{|r_n|>R_{n+1}\}.
\end{align*}
Moreover, up to slightly move $a_{n+1}$, (\ref{disjointsolutions}) in Remark \ref{rational} guarantees that $r_{n-1}=\infty$ and $r_n=a_{n+1}$ have disjoint solutions, so for $R_{n+1}$ large enough we have
$$
V_{n+1}=\{|r_{n-1}|>R_{n+1}\}\cup\{|r_n-a_{n+1}|<\epsilon_{n+1}/R_{n+1}\}=:V'_{n+1}\cup V''_{n+1}\;,
$$
with $V'_{n+1}\cap V''_{n+1}=\emptyset$. Since $|a_{n+1}|>R_n$, it follows that $R_{n+1}$ sufficiently large implies $V''_{n+1}\subseteq H_n$; the inclusions $V'_{n+1}\subseteq V_n$ and $H_{n+1}\subseteq H_n$ are trivial, so we get (\ref{deltadecr}) and (\ref{kincr}), where the strict inclusions follow from (\ref{delta}) and Lemma \ref{analysis}, which guarantees the $U_j^{(n+1)}$ shrunk arbitrarily little around the point they are neighborhood of, provided $R_{n+1}$ to be large enough.
Taking $R_{n+1}$ large enough guarantees moreover that $V''_{n+1}$ has $k_n$ components (as the equation $r_n=a_{n+1}$ has $k_n$ simple solutions by \ref{numberofsimplesolutions} in Lemma \ref{tech}) and $H_{n+1}\cap V''_{n+1}=\emptyset$ (as $R_{n+1}>|a_{n+1}|$).
Now $H_n$ has $k_n$ components, each of which splits into a component of $H_{n+1}$ and a component of $V''_{n+1}$. So the statement on the pair of pants follows.

For $n$ odd the argument is the same, just switching the role of vertical and horizontal objects.
\end{proof}


\subsection{The sequence $\{\epsilon_n\}_{n\ge1}$}\label{epsilon}
Describing constraints for $R_n$ in sections \ref{vhcomp}, \ref{sets} and \ref{pp}, we assumed to have already fixed $a_1,\dots,a_{n}$, $R_0,\dots,R_{n-1}$, and in particular $\epsilon_1,\dots,\epsilon_{n}$. Let us now see how to choose the $\epsilon_n$.

\begin{prop}\label{epsss}
There are sequences $\{a_n\}_n$, $\{\epsilon_n\}_n$ and $\{R_n\}_n$ such that 
\begin{equation}\label{eps}
\|f_n-\id\|_{\overline{B_j(\frac12)}}<\delta_j\cdot\frac1{2^{j+n}}\;\;,\hspace{2cm}0\le j< n
\end{equation} 
\end{prop}
\begin{proof} Let us construct the three sequences inductively on $n$.
Case $n=1$; take any $a_1\in\C$ such that $|a_1|>1/2$, define $B_0$ as the origin $(0,0)\in\C^2$. Then there exists $\epsilon_1>0$ such that $\|f_1-\id\|_{\overline{B_0(\frac12)}}<\delta_0\cdot\frac1{2^{1+0}}$ (note that $B_0(\frac12)$ is the $1/2$-ball in $\C^2$).\\
Assume (\ref{eps}) true for $n$. In particular $a_i,\epsilon_i,R_i,\;i\le n-1$ and $a_n,\epsilon_n$ are defined, so we can take $R_n$ to satisfy all the constraints discussed in the previous sections; in particular in Lemma \ref{tech}, Proposition \ref{gen} and \ref{prop:monotone}, condition (\ref{small}). So we can fix $a_{n+1}$ such that $|a_{n+1}|>\sqrt2R_{n}+1/2$. Define then $\epsilon_{n+1}>0$ such that $\|f_{n+1}-\id\|_{\overline{B_n(\frac12)}}<\delta_n\cdot\frac1{2^{2n+1}}$.\\
Then (\ref{eps}) follows in the case $n+1$.
\end{proof}
With such a choice of parameters, we achieve
\begin{equation}\label{forprop}
\|f_n\circ\cdots\circ f_j-\id\|_{B_{j-1}}<\frac12\;\;,\hspace{2cm}1\le j\le n\;.
\end{equation}
The case $n=1$ follows from the base case of (\ref{eps}).
Assume (\ref{forprop}) true for $n$ and let us see it holds for $n+1$:
\begin{align}
\|f_{n+1}\circ\cdots\circ f_j-\id\|_{B_{j-1}}
\le&\|f_{n+1}-\id\|_{f_n\circ\cdots\circ\ f_{j}(B_{j-1})}\nonumber\\
+&\|f_{n}-\id\|_{f_{n-1}\circ\cdots\circ\ f_{j}(B_{j-1})}
+\cdots+\|f_{j}-\id\|_{B_{j-1}}\nonumber\\
\le\label{ind}&\|f_{n+1}-\id\|_{\overline{B_{j-1}(\frac12)}}\\
+&\|f_{n}-\id\|_{\overline{B_{j-1}(\frac12)}}
+\cdots+\|f_{j}-\id\|_{\overline{B_{j-1}(\frac12)}}\nonumber\\
\le\label{prop}&\frac{\delta_{j-1}}{2^{n+1+j-1}}+\cdots+\frac{\delta_{j-1}}{2^{j+j-1}}\\
=&\delta_{j-1}\cdot\sum_{i=2j-1}^{n+j}\frac1{2^i}<\frac12\;.\nonumber
\end{align}
Hence (\ref{forprop}) holds for any $1\le j\le n+1$ as promised. Just observe that (\ref{ind}) follows since
$$
f_{n}\circ\cdots\circ\ f_{j}(B_{j-1})\subset B_{j-1}(1/2)\subset\overline{B_{j-1}(1/2)}
$$
holds true by inductive hypothesis for any $1\le j\le n$ and (\ref{prop}) follows from Proposition \ref{epsss}.

\begin{rk}
Since the three sequences underlying the whole construction depend one on the other, it is important to highlight the order with which their elements have to be taken. We have seen in sections \ref{vhcomp}, \ref{sets} and \ref{pp} that in order to define $R_n$, the parameters
$$
a_j,\epsilon_j,R_j,\;j\le n-1,\;\; \text{and}\;\;\; a_n,\epsilon_n\;
$$
need to be already fixed; then we immediately fix $a_{n+1}$ such that $|a_{n+1}|>\sqrt2R_n+1/2$.
Similarly, in section \ref{epsilon} it turned out that in order to define $\epsilon_n$, 
$$
a_j,\epsilon_j,R_j,\;j\le n-1,\;\;\text{and}\;\;a_n
$$
need to be previously fixed; in particular, to choose $\epsilon_1$, some $a_1\in\C$, $|a_1|>1/2$, has to be fixed.
Namely, the order to follow for choosing the sequences of parameters in order to perform our construction is
$$
a_1,\epsilon_1,R_1,a_2,\epsilon_2,R_2,\dots\;.
$$

\end{rk}

\vspace{0cm}

\section{The Cantor set $C$ and the proper holomorphic embedding $\gamma$}

\subsection{Definition of the Cantor set $C$} \label{defcant}

We are ready to define the object of our interest:
\newline
\begin{equation}\label{cantordefinition}
C:=\bigcap_n\Delta_n\subset\overline\C.\\
\end{equation}
\newline
Let us see $C$ is indeed a Cantor set by proving that it fulfills all the properties \ref{ne}--\ref{met} stated in section \ref{prelim}.
Being $\overline \C$ metrizable, we get \ref{met}.
Since for every $n\in\N$ one has
$$
\gamma_n^{-1}(\infty_2)\subset\gamma_{n+j}^{-1}(\infty_2)\subset\Delta_{n+j}\;\;\;\;\forall j\in\N	\;,
$$
it follows that\\
\begin{equation}\label{descr}
\bigcup_{n\ge1}\gamma_n^{-1}(\infty_2)\subseteq C\;,
\end{equation}
from which $C\neq\emptyset$ and thus we get \ref{ne}.
From the above construction, every component of $\Delta_n$ eventually shrinks to a point (see e.g. (\ref{small})) and since the role of vertical and horizontal components switches at each step, it follows that $\forall n\in\N,\;\forall j\in\{1,\dots,b_n\}$ there will be infinitely many pair of pants inside each $U_j^{(n)}$, so there will be no isolated point. Moreover exploiting (\ref{kincr}) one gets
\begin{equation}\label{closedc}
\left(\overline\C\setminus C\right)^{\circ}
=\left(\bigcup_n K_n\right)^{\circ}
=\bigcup_n K_n^{\circ}
=\bigcup_n K_n
=\overline\C\setminus C
\end{equation}
from which $C$ is closed, and thus \ref{per} is verified.
Being $C$ closed and $\overline\C$ compact, it follows that $C$ is compact itself, so we have \ref{com}.
Finally, if any connected component of $C$ had more than one point, then (\ref{small}) would fail for $n$ big enough. So $C$ is totally disconnected, that is we have \ref{td} and so we have proved $C$ is indeed a Cantor set.

\subsection{Definition of the proper holomorphic embedding $\gamma\colon\overline \C\setminus C\hookrightarrow\C^2$}\label{phe}
From time to time we will use the following fact
\begin{rk}\label{wlog}
If $\{K_n\}_n$ is a normal exhaustion for a domain $\Omega\subseteq\overline\C$ and $\{z_n\}_n\subset\Omega$ converges to $z_0\in\partial\Omega$, then we can suppose without loss of generality that $z_n\in K_n\setminus K_{n-1}$.
\end{rk}
Consider $\gamma_n$ as a mapping $K_{n+1}^{\circ}\to\C^2$; thus the natural domain to define a limit mapping is
$
\bigcup_{n\ge1}K_n^{\circ}
$
which equals $\overline\C\setminus C$ from (\ref{closedc}).
For every $n\ge k$ one has
\begin{alignat*}{5}
\|\gamma_{n+1}-\gamma_n\|_{K_k}
&=\|(f_{n+1}-\id)\circ\gamma_n\|_{\gamma_k^{-1}(B_k)}\\
&=\|(f_{n+1}-\id)\circ(f_n\circ\cdots\circ f_{k+1}\circ\gamma_k)\|_{\gamma_k^{-1}(B_k)}\\
&\le\|f_{n+1}-\id\|_{f_n\circ\cdots\circ f_{k+1}(B_k)}\\
&\le\|f_{n+1}-\id\|_{\overline{B_k(\frac12)}}&&\mbox{by (\ref{forprop})}\\
&<\delta_k\cdot\frac1{2^{n+1+k}}&&\mbox{by (\ref{eps})}\;\;,
\end{alignat*}
from which it follows that
$$
\sum_{n\ge k}\|\gamma_{n+1}-\gamma_n\|_{K_k}
<\sum_{n\ge k}\delta_k\cdot\frac1{2^{n+1+k}}<+\infty\;
$$
for every fixed $k\ge1$; this proves that $\{\gamma_n\}_n$ converges uniformly on compact subsets of $\overline\C\setminus C$, in fact this last set is $\bigcup_nK_n$ and $\{K_n\}_n$ is a normal exhaustion of $\overline\C\setminus C$.
We have thus defined a holomorphic mapping
$$
\gamma\colon\overline \C\setminus C\to\C^2\;.
$$ 
It is now easily seen that $\gamma$ is a local embedding (namely: an embedding on every $K_{k-1}$) since for every $k\ge1$ one has
$$
\|\gamma-\gamma_k\|_{K_k}\le
\sum_{n\ge k}\|\gamma_{n+1}-\gamma_n\|_{K_k}
<\sum_{n\ge k}\delta_k\cdot\frac1{2^{n+1+k}}<\delta_k\;.
$$
Finally, global embedding property of $\gamma$ will follow automatically once properness is proved, achieving thus that $\gamma$ is a proper holomorphic embedding from the complement of a Cantor set into $\C^2$, as wanted.
Consider then $\{z_n\}_n\subset\overline\C\setminus C=\bigcup_nK_n$ such that $z_n\to\partial(\overline\C\setminus C)=\partial C$. By Remark \ref{wlog} we can assume $z_n\in K_n\setminus K_{n-1}\subseteq\Delta_{n-1}$, from which $|\gamma_{n-1}(z_n)|\to+\infty$ as $n\to+\infty$. Nevertheless this does not imply that $|\gamma(z_n)|\to+\infty$, as it just claims the divergence of $\{\gamma_{n-1}(z_n)\}_n$ and no information on how the sequence $\{\gamma_n\}_n$ acts on the elements of $\{z_n\}_n$ is given. In particular, we need to achieve $z_n$ to be set outside some $R_n$-ball (say $B_{n-3}$) by a whole tail of $\{\gamma_n\}_n$ (say $\gamma_{n+j}$ for every $j\in\N$). In this way the possibility for $z_n$ to come back inside a ball is prevented and we obtain properness of the limit function $\gamma$.
Hence properness will be achieved if, for instance, the following condition is satisfied:
\begin{equation}\label{proper}
\gamma_{n+j}(K_n\setminus K_{n-1})\subset\overline\C^2\setminus B_{n-3}\hspace{2cm}\forall n\ge3,\;j\in\N \;.
\end{equation}
As we assumed $z_n\in K_n\setminus K_{n-1}\;\;\forall n$, one has
$$
\gamma(z_n)=\lim_j\gamma_{j+n}(z_n)\stackrel{(\ref{proper})}{\in}\overline\C^2\setminus B_{n-3}\;,
$$
which implies $|\gamma(z_n)|\ge R_{n-3}$ and thus $\lim_{n}|\gamma(z_n)|=+\infty$, as wanted.
Notice that
\begin{align}\label{mid}
\gamma_n(\gamma_n^{-1}(B_n))=A_n\cap B_n\;
\end{align}
holds true for every $n\in\N$ and this allows to prove that (\ref{proper}) holds, in fact
\begin{alignat*}
\alpha
\gamma_{n+j}(K_n\setminus K_{n-1})
&=f_{n+j}\circ\cdots\circ f_{n+1}\circ\gamma_n\left(\gamma_n^{-1}(B_n)\setminus\gamma_{n-1}^{-1}(B_{n-1})\right)
&&&\mbox{by}\;\;(\ref{mid})\\
&= f_{n+j}\circ\cdots\circ f_{n+1}\left((A_n\cap B_n)\setminus (f_n(B_{n-1})\cap A_n)\right)\\
&\subseteq f_{n+j}\circ\cdots\circ f_{n+1}\left(B_n\setminus f_n(B_{n-1})\right)\\
&\subseteq f_{n+j}\circ\cdots\circ f_{n+1}\left(B_n\setminus B_{n-2}\right)
&&&\mbox{by (\ref{forprop})}\\
&\subseteq B_{n+1}\setminus B_{n-3}
&&&\mbox{by (\ref{forprop})}\\
&\subseteq\overline \C^2\setminus B_{n-3}\;,
\end{alignat*}
as promised.

\subsection{On the Hausdorff dimension of $C$ and proof of Theorem \ref{main}}
Fix $\epsilon>0$; the \emph{Hausdorff $\epsilon$-measure} of $C$ is, by definition
\begin{align*}
\mathcal H^{\varepsilon}(C)
:&=\liminf_{r\to0}\left\{\inf\left\{\sum_i|T_i|^{\varepsilon}\;:\;C\subseteq\bigcup_iT_i,\;\;|T_i|<r\right\}\right\}\;.
\end{align*}
Since $C\subset\Delta_n=\bigcup_{j=1}^{b_n}U_j^{(n)}$ holds true for any $n$ and since for every $r>0$ there exists $N_r$ such that $1/b_n^{n}<r$ for all $n\ge N_r$, exploiting (\ref{small}), it trivially follows that
\begin{align*}
\inf\left\{\sum_i|T_i|^{\varepsilon}\;:\;C\subseteq\bigcup_iT_i,\;\;|T_i|<r\right\}
&\le\inf\left\{\sum_{j=1}^{b_n}|U_j^{(n)}|^{\varepsilon}\;:\;\;n\ge N_r\right\}\\
&\le\inf\left\{\frac1{b_n^{\epsilon n-1}}\;:\;\;n\ge N_r\right\}=0\;,
\end{align*}
hence $\mathcal H^{\varepsilon}(C)=0$ for any $\epsilon>0$ fixed. Therefore, recalling that the \emph{Hausdorff dimension} of $C$ is defined as
$$
\dim_H(C):=\inf\{d\ge0\;:\;\mathcal H^d(C)=0\}\;,
$$
we conclude that
$$
\dim_H(C)=0\;.
$$
This last observation, together with the construction of the proper holomorphic embedding $\gamma:\overline \C\setminus C\hookrightarrow\C^2$ made in section \ref{phe}, completes the proof of Theorem \ref{main}.

\subsection{Two further characterizations of $C$}
The original definition of the Cantor set $C$ given in \cite{O} is
\begin{equation}\label{orevkovdef}
C=\gamma^{-1}(\infty_2)\;.
\end{equation}
Actually $\gamma$ is not defined on the Cantor set $C$, therefore, in order for this last relation to make sense, we need to extend $\gamma$ on $C$. Being $C$ a Cantor set in $\overline{\C}$, it is the boundary of $\overline{\C}\setminus C$, domain on which $\gamma$ is proper, therefore
$$
\lim_{w\to C,w\in \overline{\C}\setminus C}|\gamma(w)|=\infty\;,
$$
in fact if $z\in C$, being $C$ a Cantor set,  $z$ is an accumulation point of the complement of $C$, which is $\bigcup_nK_n$, thus $z$ can be hit by a sequence $\{z_n\}_n\subset \bigcup_nK_n$; by Remark \ref{wlog} let $z_n\in K_n\setminus K_{n-1}$. On one hand it is clear that $\lim_j\gamma_{n+j}(z_n)=\gamma(z_n)$ for every $n$, on the other hand (\ref{proper}) implies $\lim_n\lim_j\gamma_{n+j}(z_n)=\infty_2\;$. So we can extend $\gamma$ on $C$ by continuity defining $\gamma\equiv\infty_2$ on $C$. Since trivially $\gamma\neq\infty_2$ on $\overline{\C}\setminus C$, we have that defining $C$ as in (\ref{orevkovdef}) matches with the definition (\ref{cantordefinition}) given in Section \ref{defcant}. We conclude by characterizing $C$ in one last way:
\begin{equation}\label{lastdefinition}
C=\cl\left(\bigcup_{n\ge1}\gamma_n^{-1}(\infty_2)\right)\;.
\end{equation}
Being $C$ closed, (\ref{descr}) implies that
$$
\cl\left(\bigcup_{n\ge1}\gamma_n^{-1}(\infty_2)\right)\subseteq C\;.
$$
Consider now the following inductive procedure: at step $n$ we have $b_n$ points, the elements of $\gamma_n^{-1}(\infty_2)$, and $b_n$ neighborhoods $U_j^{(n)}$ of them. At step $n+1$ we have the old $b_n$ points and $k_{n}$ new points (that is, the elements of $\gamma_{n+1}^{-1}(\infty_2)$) and $b_{n+1}=b_n+k_{n}$ neighborhoods $U_j^{(n+1)}$ of them, $b_n$ of which are just the $U_j^{(n)}$ shrunk around the points of $\gamma_{n}^{-1}(\infty_2)$ and the $k_{n}$ remaining ones are neighborhoods of the new points $\gamma_{n+1}^{-1}(\infty_2)\setminus \gamma_{n}^{-1}(\infty_2)$.
Now $\cup_{j=1}^{b_n}U_j^{(n)}=\Delta_n$ forms a decreasing sequence of open sets, whose intersection defines $C$ and all the $U_j^{(n)}$ eventually shrink to the point they are neighborhood of, hence $C$ is the limit of the above inductive procedure.
Therefore, given $w\in C$, either it appears at some finite step of the procedure (hence $w\in\gamma_n^{-1}(\infty_2)$ for some $n$), or it appears in the limit: by construction, the points tend to accumulate (in fact there is no isolated point, see \ref{per}), that is $w\in\operatorname{Acc}\left(\bigcup_{n\ge1}\gamma_n^{-1}(\infty_2)\right)\setminus\left(\bigcup_{n\ge1}\gamma_n^{-1}(\infty_2)\right)$. So $w\in\cl\left(\bigcup_{n\ge1}\gamma_n^{-1}(\infty_2)\right)$ and we are done.

\newpage
\bibliographystyle{amsplain}

\end{document}